\documentclass[pdflatex,sn-mathphys-num]{sn-jnl}
\usepackage{graphicx} 
\usepackage{xcolor}
\usepackage{amsmath}
\usepackage{amssymb,amsfonts}
\usepackage{fullpage}
\usepackage{algorithm}
\usepackage{algpseudocode}
\usepackage{siunitx}
\usepackage{orcidlink}
\usepackage{aliascnt}

\usepackage{multirow}%
\usepackage{mathrsfs}%
\usepackage{hyperref}

\usepackage[title]{appendix}%
\usepackage{xcolor}%
\usepackage{textcomp}%
\usepackage{manyfoot}%
\usepackage{booktabs}%
\usepackage{algorithmicx}%
\usepackage{listings}%

\newtheorem{theorem}{Theorem}[section]

\newaliascnt{lemma}{theorem}
\newtheorem{lemma}[lemma]{Lemma}
\aliascntresetthe{lemma}

\theoremstyle{definition}
\newaliascnt{definition}{theorem}
\newtheorem{definition}[definition]{Definition}
\aliascntresetthe{definition}

\newcommand{\R}{\mathbb{R}}
\newcommand{\C}{\mathbb{C}}

\newcommand{\bs}{\boldsymbol}

\usepackage[nameinlink,capitalize,noabbrev]{cleveref}

\begin{document}

\title[Article Title]{On high probability of universal approximation in random basis expansions with non-continuous weight sampling}

\author*[1]{\fnm{John E.} \sur{Darges\orcidlink{0009-0007-9059-8921}}}\email{jdarges@emory.edu}

\affil*[1]{\orgdiv{Department of Mathematics}, \orgname{Emory University}, \orgaddress{ \city{Atlanta, Georgia 30322}, \country{United States}}}

\abstract{Random basis expansion (RBE) search the span of a randomly sampled basis to find the best approximation of a target function. They are equivalent to single layer neural networks where the hidden layer weights are chosen randomly. Universal approximation properties have been established for RBEs using continuous weight sampling distributions and real-valued activation functions. Our results extend the universal approximation property to RBEs that use non-continuous weight distributions with dense support in the weight space and that use complex-valued activation functions. The result shows such random bases have the universal approximation property with arbitrarily high probability.}


\maketitle

\section{Introduction}
The random basis expansion (RBE) is a tool for representing a target function $f:\C^n\to \C$ in a randomly sampled basis. Compare this against representing a function in a standard ordered basis such as the Fourier basis or an orthogonal polynomial basis.
The concept of a RBE appears various names and forms. These include random feature expansion~\cite{RahimiRecht08,Hashemi23}, neural network with random weights~\cite{broomhead1988multivariable}, random vector functional link~\cite{PaoPark94}, and extreme learning machine~\cite{huang2006universal}. Such an expansion takes the form
\begin{equation}
    \widehat{f}(\bs x) = \sum_{j=1}^m\alpha_j\phi(\bs\omega_j^\top\bs x+b_j),\quad \bs\omega_j\sim\pi,b_j\sim\nu.
\end{equation}
Here, $\phi$ is called the activation function or feature map. The weights $\bs\omega_j$ and biases $b_j$ are independent samples from probability distributions $\pi$ and $\nu$, respectively. Each $\phi(\bs\omega_j^\top\bs x+b_j)$ is termed a feature or random basis function. Unlike in a fully-trained neural network, only the coefficients $\alpha_j$ are tuned to yield an approximation of $f$.
The earliest uses of random bases can be attributed to~\cite{broomhead1988multivariable,Schmidt92} and developed into random vector functional link by~\cite{PaoPark94,PaoIngelnik95} and extreme learning machine~\cite{huang2006universal}.
Random feature expansions emerged independently, in~\cite{RahimiRecht07,RahimiRecht08}, out of the lineage of kernel methods.

Universal approximation properties of RBE jusify their use. A basis expansion has the universal approximation property if any continuous function can be approximated arbitrarily well in the span of said basis. This has been shown for randomly sampled bases in the case where the weight sampling distribution $\pi$ is a continuous one in
~\cite{huang2006universal,Needell22}.
It has been proven under certain boundedness conditions on the activation function and higher moments of the weight sampling distribution in~\cite{Sun19}. Random feature expansion literature emphasizes the universal approximation less, instead justifying through concentration inequalities on the error for target functions coming from a corresponding reproducing kernel Hilbert space. See, for example,~\cite{RahimiRecht08,Saha22,Hashemi23,weidensager2025anova}.

There is recent interest in random basis methods where weights are sampled from a non-continuous distribution. Sparse weight methods~\cite{Hashemi23,darges24,weidensager2025anova} treat weights as a product of continuous and discrete random variables, giving a spike-and-slab weight distribution. Current universal approximation theorems do not cover this case. Random Fourier feature methods rely on a complex-valued activation function, also not considered for universal approximation by random bases. 
We give a proof of universal approximation for random bases with a complex-valued activation function and weight distribution $\pi\times\nu$ that is not continuous, but has dense support in $\R^{n+1}$. The theory is not meant to be constructive. It shows that universally approximating random basis expansions from non-continuous distributions exist.

\section{Universal approximation preliminaries}
The notion of universal approximation is the standard one that appears in in~\cite{Leshno93,Pinkus99}.
\begin{definition}\label{def:univ}
    Let $\mathcal{A}$ denote a collection of basis functions. The collection $\mathcal{A}$ satisfies the universal approximation property if, 
     for any continuous function $f\in\mathcal{C}(\R^n)$ and any compact subset $K\subset\R^n$, there exists $\widehat{f}\in\text{span}(\mathcal{A})$ for any choice of $\epsilon>0$ so that 
     \begin{equation}
     \|f-\widehat{f}\|_{L_\infty(K)}<\epsilon.
     \end{equation}
\end{definition}
In other words, for every compact domain, the span of the basis is dense in the set of continuous functions. This is also known as density on compacta for continuous functions.
A landmark result from Pinkus~\cite{Pinkus99} shows that fully-trained networks satisfy the universal approximation property.
\begin{theorem}[Pinkus~\cite{Pinkus99}]\label{thm:pin_univ}
   Let $\mathcal{A} = \{\phi(\bs\omega^\top \bs x+b):\,\bs\omega\in\R^n,\ b\in\R\}$, with $\phi:\R\to\R$ an activation function. Then $\mathcal{A}$ satisfies ~\cref{def:univ} if and only if $\phi$ is not a polynomial.
\end{theorem}
This theorem is a key lemma for our main result. To handle the case of complex activation functions, we need a generalization of~\cref{thm:pin_univ}. Such a theorem for approximation of complex-valued continuous functions also exists~\cite{voigtlaender2023universal}.
\begin{theorem}[Voigtlaender~\cite{voigtlaender2023universal}]\label{thm:voi_univ}
   Let $\mathcal{A} = \{\phi(\bs\omega^\top \bs x+b):\,\bs\omega\in\C^n,\ b\in\C\}$, where $\phi:\C\to\C$ is locally bounded and continuous almost everywhere. Then $\mathcal{A}$ satisfies the universal approximation property~\cref{def:univ} if and only if $\phi$ is not almost polyharmonic.
\end{theorem}
Being not almost polyharmonic is analogous to being not a polynomial. A function $\phi$ is almost polyharmonic if $\phi=g$ almost everywhere with $\Delta^mg=0$, where $\Delta^m$ is a higher-order Laplace operator. We replace~\cref{thm:pin_univ} with~\cref{thm:voi_univ} for the complex case.
Both theorems are non-constructive results that guarantee the existence of \textit{some} expansion that arbitrarily approximations $f$.
We prove a lemma before we articulate the main universal approximation result. The lemma relates the distance between a pair of weights to the distance between  the basis functions corresponding to those weights.
\begin{lemma}\label{lem:arzela}
Let $\phi:\R\to\C$ be continuous. 
Let $\{(\bs \omega_j,b_j)\}_{j=1}^\infty\subset\R^{n+1}$ be a sequence 
that converges to $(\bs \omega^*,b^*)\in\R^{n+1}$. Then, for any compact subset $K\subset\R^n$,
\begin{equation}\label{equ:lem_arz}
\lim_{j\to\infty}\|\phi_j-\phi^*\|_{L_\infty(K)} = 0,
\end{equation}
where $\phi_j(\bs x) = \phi(\bs{\omega}_j^\top\bs{x}+b_j)$ and $\phi^*(\bs x) = \phi({\bs{\omega}^*}^\top\bs{x}+b^*)$.
\end{lemma}
\begin{proof}
We will proceed by showing that $\{\phi_j\}_{j=1}^\infty$ is equicontinuous and uniformly bounded and use 
the Arzel\`{a}-Ascoli theorem to obtain the desired result. 

Fix $\epsilon > 0$. Because $\{(\bs \omega_j,b_j)\}_{j=1}^\infty$ is a convergent sequence, it is contained in some 
compact set $Q\subset\R^{n+1}$. We then denote
\begin{equation}
M = \max_{\begin{subarray}{l} (\bs \omega,b)\in Q,\\ 
   \ \bs x\in K\end{subarray}} |\phi(\bs \omega^\top\bs x + b)|.
\end{equation} 
As $\phi$ is continuous, it is uniformly continuous on the compact set $T\subset\C$ bounded by $M$. For every $t,\hat t\in T$, 
there exists $\delta^*>0$ so that $|t-\hat t|<\delta^*$ implies $|\phi(t) - \phi(\hat t)|<\epsilon$.

We now show equicontinuity. First, 
\begin{equation}
\sup_{j}\|\bs \omega_j\| =  \sup_{j}\|(\bs \omega_j,0)\| \leq \max_{(\bs \omega, b)\in Q}\|(\bs \omega, b)\|.
\end{equation}
Let $\alpha = \max_{(\bs \omega, b)\in Q}\|(\bs \omega, b)\|$. Then, for any $j>0$, any $\bs x_1,\bs x_2\in\R^n$, 
and our same fixed $\epsilon>0$, let $\delta < \frac{\delta^*}{\alpha}$. Using the Cauchy-Schwartz inequality,
$\|\bs x_1- \bs x_2\|<\delta$ implies 
\begin{equation}
   \begin{aligned}
|{\bs{\omega}_j}^\top\bs x_1 + b - ({\bs{\omega}_j}^\top\bs x_2 + b)| & = |{\bs{\omega}_j}^\top(\bs x_1 - \bs x_2| \\
 \leq \|\bs \omega_j\|\cdot\|\bs x_1 - \bs x\| &
  < \|\bs{\omega}_j\| \delta < \frac{\|\bs{\omega}_j\|}{\alpha} \delta^* < \delta^*.
\end{aligned}
\end{equation}
Letting $t = \bs \omega_j^\top\bs x_1 + b$ and $\hat{t}= \bs \omega_j^\top\bs x_2 + b$, this implies 
$|\phi(t) - \phi(\hat{t}
)|<\epsilon$. Then $\{\phi_j\}_{j=1}^\infty$ is equicontinuous. Moreover, 
the sequence is uniformly bounded by $M$. The Arzel\`{a}-Ascoli theorem then implies 
that there exists a uniformly convergent subsequence. However, the continuity of $\phi$ implies
that $\{\phi_j\}_{j=1}^\infty$ converges pointwise to $\phi^*$. Therefore, $\{\phi_j\}_{j=1}^\infty$
converges uniformly to $\phi^*$. Since all $\{\phi_j\}_{j=1}^\infty$ and $\phi^*$ are continuous, 
uniform convergence on $K$ implies convergence in $L_\infty(K)$.
\end{proof}

\section{Universal approximation with high probability}
When using a distribution with dense support in $\R^{n+1}$, we are highly likely to sample basis functions that are arbitrarily close to those that appear in expansions guaranteed by~\cref{thm:pin_univ} and~\cref{thm:voi_univ}.
Relying on~\cref{lem:arzela} the below result shows this.
\begin{theorem}\label{thm:my_univ}
Let $f:\C^n\to\C$ be a continuous function and let the activation function be a  continuous non-polyharmonic function $\phi:\R\to\C$.
For every compact subset $K\subset\C^n$ and every $\epsilon>0$, the following holds: 
when $m\geq\log\Big(\frac{1}{\big(1-(1-\eta)^{1/q}\big)^q}-\frac{1}{1-C}\Big)$, then 
\begin{equation}\label{equ:thm_my}
   \Big\|f-\sum_{j=1}^{m}\alpha_j\phi_j\Big\|_{L_\infty(K)}<\epsilon,
\end{equation}
holds with probability $1-\eta$, where $\phi_j(\bs x) = \phi(\bs \omega_j^\top \bs x + b_j)$ and 
the coefficients $\{\alpha_j\}_{j=1}^{m}$ depend on $(\bs \omega_j,b_j)_{j=1}^{m}$. 
Here, each $(\bs \omega_j,b_j)_{j=1}^{m}$ is sampled independently and identically from a probability 
distribution $\mathcal{D}$ whose support is dense in $\R^{n+1}$. The constant 
$C=C(f,\epsilon,\mathcal{D})$ depends on 
$\epsilon$, $f$, and the sampling distribution $\mathcal{D}$. The constant $q=q(f,\epsilon)$ 
depends on $\epsilon$ and $f$.
\end{theorem}
\begin{proof}
Given $\epsilon>0$, universal approximation properties of $\phi$ guarantee that for some $q>0$, there exists an expansion that
approximates $f$ within $\epsilon/2$,
\begin{equation}
   \Big\|f-\sum_{i=1}^q\tilde{\alpha}_i\tilde{\phi}_i\Big\|_{L_\infty(K)}<\frac{\epsilon}{2},
\end{equation}
where $\tilde{\phi}_i(\bs x) = \phi(\tilde{\bs{\omega}}_i^\top\bs x + \tilde{b}_i)$. When $f:\R^n\to\R$ and $\phi:\R\to\R$, this follows from~\cref{thm:pin_univ}. If $f:\C^n\to\C$ and $\phi:\R\to\C$, this follows from~\cref{thm:voi_univ}.
Let $\alpha^* = \max_i\tilde{\alpha}_i$. Now,~\cref{lem:arzela} 
implies that some $\delta_i$ exists for each $\tilde{\phi}_i$ so that $\|(\bs{\omega},b) - (\tilde{\bs{\omega}}_i,\tilde{b})\|<\delta_i$
means that $\|\phi^* - \tilde{\phi}_i\|_{L_\infty(K)}<\frac{\epsilon}{2q\alpha^*}$ for $\phi^*(\bs x) = \phi(\bs{\omega}^\top\bs x + b)$, with
$(\bs{\omega},b)\in\R^{n+1}$. 

For each $i$, consider a collection of $l$ independent samples $\{(\bs{\omega}_k,b_k)\}_{k=1}^l$ from $\mathcal{D}$. 
Let $\mathbb{P}_{\mathcal{D}}$ denote the probability measure of $\mathcal{D}$. 
The probability that a sample from $\mathcal{D}$ is in the  $\delta_i$-neighborhood of $(\tilde{\bs{\omega}}_i,\tilde{b})$ is 
$v_i=\mathbb{P}_{\mathcal{D}}\big(B_{\delta_i}(\tilde{\bs{\omega}}_i,\tilde{b})\big)$. Let $v^* = \max_i v_i$. 
The probability that at least 
one random $(\bs{\omega}_k,b_k)$ belongs to the $\delta_i$-neighborhood of $(\tilde{\bs{\omega}}_i,\tilde{b})$ is 
\begin{equation}
\begin{aligned}
1 - \prod_{k=1}^l \mathbb{P}_{\mathcal{D}}(\|(\bs{\omega}_k,b_k) -(\tilde{\bs{\omega}}_i,\tilde{b})\|\geq\delta_i) 
& = 1 - \prod_{k=1}^l (1 - v_i) = 1 - (1- v_i)^l\\
& \leq 1 - (1 - v^*)^l.
\end{aligned}
\end{equation}
Some random $(\bs{\omega}_k,b_k)$ is in each $\delta_i$-neighborhood
 with probability at least $1-\eta$, if $q,l$ satisfy
\begin{equation}
   \big(1 - (1 - v^*)^l\big)^q \geq 1 - \eta.
\end{equation}
To obtain a sufficient bound based on the total number of sampled weights, $m = q\cdot l$, we require $m$ so that
\begin{equation}
\big(1 - (1 - v^*)^{m/q}\big)^q \leq 1 - \eta.
\end{equation}
This is satisfied when
\begin{equation}
m \geq \log\Big(\frac{1}{\big(1-(1-\eta)^{1/q}\big)^q} - \frac{1}{1-v^*}\Big).
\end{equation}
We choose the coefficients $\{\alpha_j\}_{j=1}^m$ in the following manner: for $k=1,\dots,l$ and $i=1,\dots,n$,
consider the collection of $l$ samples $\mathcal{J}_i = \{(\bs{\omega}_j,b_j)\}_{j = (i-1)l + 1}^{i\cdot l}$ corresponding to 
$i = 1,\dots, q$. For each $\mathcal{J}_i$, 
if $(\bs{\omega}_j,b_j)$ has the minimal distance to $(\tilde{\bs{\omega}}_i,\tilde{b})$ over $\mathcal{J}_i$, set $\alpha = \tilde{\alpha}_i$. 
Otherwise, set $\alpha_j = 0$.

Consider $\sum_{j=1}^m\alpha_j\phi_j$ where $\phi_j(\bs x) = \phi(\bs\omega_j^\top\bs x + b_j)$. When 
$m \geq \log\Big(\frac{1}{\big(1-(1-\eta)^{1/q}\big)^q} - \frac{1}{1-v^*}\Big)$, we have with probability at least
$1-\eta$ that 
\begin{equation}\label{equ:mythm_bound1}
\begin{aligned}
\Big\|f - \sum_{j=1}^m\alpha_j\phi_j\Big\|_{L_\infty(K)} & = \Big\|f -
 \sum_{i=1}^q\tilde{\alpha}_i\tilde{\phi}_i + \sum_{i=1}^q\tilde{\alpha}_i\tilde{\phi}_i -\sum_{j=1}^m\alpha_j\phi_j\Big\|_{L_\infty(K)} \\
 & \leq \Big\|f -
 \sum_{i=1}^q\tilde{\alpha}_i\tilde{\phi}_i \Big\|_{L_\infty(K)} + 
 \Big\|\sum_{i=1}^q\tilde{\alpha}_i\tilde{\phi}_i -\sum_{j=1}^m\alpha_j\phi_j\Big\|_{L_\infty(K)}  \\
 & < \frac{\epsilon}{2} +  \Big\|\sum_{i=1}^n\tilde{\alpha}_i\tilde{\phi}_i -\sum_{j=1}^m\alpha_j\phi_j\Big\|_{L_\infty(K)} .
\end{aligned}
\end{equation}
Let $\{j_i\}_{i=1}^n$ denote the indices corresponding to nonzero coefficients $\alpha_{j_i} = \tilde{\alpha}_i$. 
Then, 
\begin{equation}\label{equ:mythm_bound2}
   \begin{aligned}
  \Big\|\sum_{i=1}^q\tilde{\alpha}_i\tilde{\phi}_i -\sum_{j=1}^m\alpha_j\phi_j\Big\|_{L_\infty(K)} = &
\Big\|\sum_{i=1}^q(\tilde{\alpha}_i\tilde{\phi}_i - \alpha_{j_i}\phi_{j_i})\Big\|_{L_\infty(K)} \\
  =  \Big\|\sum_{i=1}^q\tilde{\beta}_i(\tilde{\phi}_i - \phi_{j_i})\Big\|_{L_\infty(K)} 
  &\leq  \sum_{i=1}^q \tilde{\beta}_i\|\tilde{\phi}_i - \phi_{j_i}\|_{L^\infty(K)}\\
  <  \sum_{i=1}^q \tilde{\alpha}_i \frac{\epsilon}{2q\alpha^*} \leq \frac{\epsilon}{2}&.
   \end{aligned}
\end{equation} 
Then~\eqref{equ:mythm_bound1} and~\eqref{equ:mythm_bound2} imply that 
\begin{equation}
   \Big\|\|f - \sum_{j=1}^m\alpha_j\phi_j\Big\|_{L_\infty(K)} < \frac{\epsilon}{2} + \frac{\epsilon}{2} =  \epsilon.
\end{equation}
\end{proof}
A non-continuous  distribution can yield a collection of random basis functions whose span is dense in the span of all basis functions of the same form. The well-established universal approximation theorems for fully-trained networks guarantee density in the collection of continuous functions.

\section{Conclusion}\label{sec:conclusion}
Random basis expansions (RBEs) approximate a target function by a weighted sum of feature functions evaluated at randomly sampled weights and biases, with only the coefficients tuned to fit the data. Existing universal approximation results for RBEs require the weight-sampling distribution to be continuous, a condition violated by the sparse, spike-and-slab-type sampling schemes. We extend the universal approximation property to weight non-continuous distributions withe dense support in $\R^{n+1}$, covering real-valued activation functions that are not polynomials and complex-valued activation functions that are not almost polyharmonic. We show that an RBE with 
randomly sampled features approximates any continuous function on a compact domain to within any prescribed accuracy with high probability, depending on the target function, the desired accuracy, and the sampling distribution. The result shows existence, establishing that universally approximating random basis expansions from non-continuous weight distributions exist.

\bmhead{Acknowledgments} This work is partially supported by  by the National Science Foundation (NSF) under grant number 2038118.

\bibliography{refs}
\end{document}